\documentclass[sn-mathphys]{sn-jnl}

\jyear{2021}%

\theoremstyle{thmstyleone}%
\newtheorem{theorem}{Theorem}
\newtheorem{proposition}[theorem]{Proposition}%
\newtheorem{corollary}[theorem]{Corollary}%
\newtheorem{definition}[theorem]{Definition}%
\newtheorem{lemma}[theorem]{Lemma}%
\newtheorem{example}[theorem]{Example}%
\newtheorem{remark}[theorem]{Remark}%

\begin{document}

\title[Modules in which pure submodule is essential in a direct summand]{Modules in which pure submodule is essential in a direct summand}

\author[1]{\fnm{Kaushal} \sur{Gupta}}\email{kaushal.gupta.rs.mat18@itbhu.ac.in}
\author[2]{\fnm{Shiv } \sur{Kumar}}\email{shivkumar.rs.mat17@itbhu.ac.in}
\author*[3]{\fnm{Ashok Ji} \sur{Gupta}}\email{agupta.apm@itbhu.ac.in}

\affil*[1,1]{\orgdiv{Department of Mathematical Sciences}, \orgname{Indian Institute of Technology(BHU)}, \orgaddress{\street{} \city{Varanasi}, \postcode{221005}, \state{Uttar Pradesh}, \country{India}}}


\abstract{In this paper, we study the class of modules have the property that every pure submodule is essential in a direct summand. These modules are termed as pure extending modules which is a proper generalisation of extending modules. Examples and counterexamples are given. We study some properties of pure extending modules and characterize regular ring, semisimple ring, local ring and  PDS ring in terms of pure extending modules.}




\maketitle

\section{Introduction}\label{sec1}
Utumi\cite{YU} observed $C_1$ condition on a ring  which is satisfied if the ring is self injective. Later, similarly $C_1$ condition on a module $M$ is defined as, every submodule of a module  $M$ is essential in a direct summand of $M$. A module satisfies $C_1$ condition known as  $C_1$ module or extending module. As we know that every submodule of a module $M$ need not be pure submodule ( Infact no submodule is pure submodule of $\mathbb{Z}$ as $\mathbb{Z}$ module). Motivated by the above facts, the objective of this paper is to extend the theory of extending modules to pure extending modules by using the concept of purity. Pure extending modules are those modules in which every pure submodule is essential in a direct summand. So pure extending modules is a proper generalisation of extending modules i.e. the class of pure extending modules is larger than the class of extending modules.
\\
 Notion of pure submodules defined by some authors in different aspects which are listed below:\\
$(i)$ P.M. Cohn \cite{PMC} called a submodule $P$ of a right $R$-module $M$ to be a pure submodule, if for each left $R$-module $N$, the sequence $0\rightarrow P\otimes N\rightarrow M\otimes N$ is exact whenever the sequence $0\rightarrow P\rightarrow M$ is exact . \\
$(ii)$ According to Anderson and Fuller \cite{AF}, a submodule $P$ of a module $M$ is said to be pure if $\mathcal{I}P=P\cap \mathcal{I}M$ for every ideal $\mathcal{I}$ of $R$. \\
$(iii)$ In \cite{PR}, Ribenboim defined a pure submodule  $P$ of a module $M$ if for every $r\in R$, $rP=rM \cap P$. In particular, $P$ is known as $RD$ (relatively divisible) pure submodule of $M$.\\
  In above definitions, $(i)\Rightarrow (ii) \Rightarrow (iii)$ but converse need not be true, while in case of  $M$ to be flat module all are equivalent. A right $R$-module $M$ is called flat if whenever $0 \rightarrow N_1 \rightarrow N_2$ is exact for left $R$-modules $N_1$ and $N_2$ then $ 0 \rightarrow M \otimes N_1 \rightarrow M\otimes N_2$ is also exact.
  A module $M$ is said to be pure $C_2$ module if every pure submodule of $M$ that is isomorphic to a direct summand of $M$ is itself a direct summand of $M$\cite{SK}. An $R$-module $M$ is said to be pure $C_3$ module if $K$ and $L$ are disjoint direct summands of $M$ then $K\oplus L$ is a pure submodule of $M$ iff $K\oplus L$ is a direct summand of $M$ \cite{LM}. \\
In section 2, after defining the notion of pure extending modules, several examples and counter examples are given to distinguish this class of modules with the various classes of modules. We show that pure extending module is a proper generalisation of extending module by giving counter examples of pure extending module that are not extending module. We provide a sufficient condition under which pure extending module implies extending module. We prove that a direct summand and a pure submodule of  pure extending module are pure extending. Let $P$ and $Q$ be $R$-modules, then a module $P$ is $Q$-pure injective if for every pure submodule $L$ of $Q$, $f\in Hom_R(L, P)$ can be extended to $g\in Hom_R(Q,P)$ such that $goh=f$ where $h\in Hom_R(L,Q)$. Further, $P$ is said to be a quasi-pure-injective if $P$ is a $P$-pure-injective, while $P$ is called a pure-injective if it is $Q$-pure injective for every $R$-module $Q$ (\cite{FH},\cite{AH},\cite{RW}).Here, we show that pure quasi injective module implies pure extending module.
 In general the following chain holds,
 $
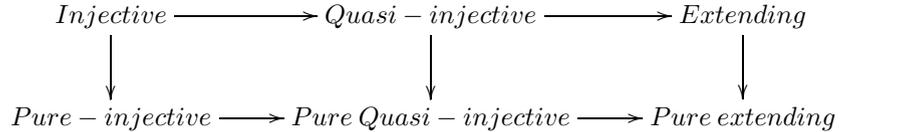
\xymatrix
 {Injective \ar[d] \ar[r] & Quasi-injective \ar[d] \ar[r] & Extending \ar[d] \\   
 Pure-injective \ar[r] & Pure\: Quasi-injective \ar[r] & Pure \: extending}$

 but converse of the above chain need not be true (see \cite{FH}, \cite{TY},\cite{RW} and Example \ref{abc}).\
 Also, we show that direct sum of pure extending modules is pure extending (see Proposition \ref{dss}).We call a module $M$, $RD$-pure extending if every $RD$-pure submodule of $M$ is essential in a direct summand of $M$. As we have seen above that not every $RD$-pure submodule is pure submodule. We show that every $RD$-pure extending module is extending module and converse need not true.\\
In section 3, we charaterize regular ring, semisimple ring, local ring and  PDS ring in terms of pure extending modules. Every pure extending  module is flat over regular ring (see Proposition \ref{fe1}). We find the equivalent conditions of pure $C_i$ for $i=1,2,3$ to be projective modules over semisimple ring (see Proposition \ref{fe2}).  A module $C$ is called cotorsion if $Ext_R ^1(F,C)=0$ for any flat module $F$ \cite{EE}. We show that flat cotorsion module is pure extending module ( see Proposition \ref{cm}) .\\
Throughout this article, we consider all rings to be associative with unity and all modules are right unital unless otherwise specified. The notations $N\leq M$, $N\leq ^{\oplus }M$ and $N\leq ^{e} M$ will denote $N$ is a submodule of $M$,  $N$ is a direct summand of $M$ and $N$ is an essential submodule of $M$ respectively. A regular ring means to be von Neumann regular ring. $E(M)$ and $PE(M)$ denote the injective hull and pure injective hull of a module $M$ respectively. For undefined terms and notions, please refer to \cite{AF}.
\section{Pure Extending Modules}
In this section, we study extending modules in terms of purity by taking pure submodules. Now we introduce pure extending modules.
\begin{definition}
   An $R$-Module $M$ is called pure extending ( or pure $C_1$), if every pure submodule of $M$ is essential in a direct summand of $M$.
  \end{definition}
     
\begin{example}
\begin{enumerate}
    \item [(i)] Since, only pure submodules of $\mathbb{Z}$-module $\mathbb{Z}$ are $\{0\}$ and $\mathbb{Z}$ itself. Therefore $\mathbb{Z}$ as a $\mathbb{Z}$-module is a pure extending. 
    \item[(ii)] Any pure injective module $M$ is pure extending module (since every pure injective module is pure quasi injective and by proposition \ref{prop1}).
    \item [(iii)]  Semisimple modules and injective modules are pure extending.
    \item [(iv)] Any finitely generated module over a Noetherian ring is pure extending module. Since its pure submodules are just direct summands [Corollary 4.91,\cite{TY}]. In particular, every finitely generated $\mathbb{Z}$-module is pure extending.
\end{enumerate} 
\end{example}

We observed that every $R$-module (in particular $\mathbb{Z}$-module) need not be pure extending.
 \begin{example}
     Consider a $\mathbb{Z}$-module $M$ such that $M=\Pi_{p\in P}\mathbb{Z}/<p>$ where $p$ varies through all primes.
     Let $P=\bigoplus_{p\in P} \mathbb{Z}/<p>$ be pure submodule which is not essential in a direct summand.
 \end{example}
     \begin{proposition} \label{psm}
    Pure submodule of pure extending module is pure extending.
    \end{proposition}
    \begin{proof}
    Let $M$ be a pure extending module and $P$ be a pure submodule of $M$. Consider $K$ be a pure submodule of $P$. Thus from [Proposition 7.2,\cite{FH}], $K$ is pure submodule of $M$. So there exists a direct summand $L$ of $M$ such that $K\leq ^e L$ which implies $K\leq ^e L\cap P$. Since $L\leq^\oplus M$, so $M=L\bigoplus L'$ for some submodule $L'$ of $M$. Now, $M\cap P=(L\bigoplus L')\cap P$ which implies $P=(L\cap P)\bigoplus (L'\cap P)$. Hence $P$ is pure extending.
    \end{proof}
    \begin{remark} Submodule of a pure extending module need not be pure extending.\\
    Let $N$ be a module  which is not a pure extending and  $PE(N)$ be pure injective hull of $N$. Then $PE(N)$ be pure injective module which implies  $PE(N)$ is pure  extending while $N\leq{PE(N)}$ is not pure extending. While over regular ring, every submodule of a pure extending module is pure extending.
    \end{remark}
 \begin{proposition} \label{DS1.1}
    Direct summand of a pure extending module  is pure extending.
\end{proposition}
\begin{proof} 
    Let $N$ be a direct summand such that $M=N\oplus N'$ of a pure extending module $M$ and $P$ be a pure submodule of $N$.  Since $P\leq N \leq^{\oplus} M$, $P$ be a pure submodule of $M$.  As $M$ is pure extending module, therefore $P$ is essential in a direct summand of $M$. Since $P\leq N$ and $P\cap N' = 0 $ then $P$ is essential in a direct summand of $N$.
\end{proof}
 The following examples justifies that the class of pure extending modules is a proper generalisation of extending.
 \begin{example} \label{abc}
     Consider $M= \mathbb{Z}_{p}\oplus \mathbb{Z}_{p^3}$ as $\mathbb{Z}$ module, where $p$ be any prime. In particular for  $p=2$, $P=\mathbb{Z}_{2}\oplus \mathbb{Z}_{8}$ as $\mathbb{Z}$ module.  As its each pure submodule is direct  summand [Corollary 4.91,\cite {TY}], so $P$ is pure extending whereas $P$ is not extending. In fact,  its submodule $\mathbb{Z}(1+2\mathbb{Z},2+8\mathbb{Z})$ is not essential in any direct summand of $P$. 
 \end{example}
 \begin{example}
      Let $R={\begin{pmatrix}
    \mathbb{Z} & \mathbb{Z}\\0 & \mathbb{Z} \end{pmatrix}}$, then $R_R$ is finitely generated and noetherian module so it is pure extending module whereas it is not extending module [Example 6.2,\cite{CAW}].
 \end{example}
 In the following proposition, we give a sufficient condition for the pure extending modules to be extending modules.
\begin{proposition} A ring $R$ is regular iff every pure extending right (resp.,left ) $R$-module is extending module.
\end{proposition} 
\begin{proof}
         Let $M$ be  pure extending module then every pure submodule of $M$ is essential in a direct summand $M$. As $R$ is a von-Neumann regular ring so every submodule of an $R$-module $M$ is pure. Hence $M$ is a extending module.\\
        Conversely, every pure extending right (resp., left ) $R$-module is extending module,  that implies every submodule of $R$-module $M$ is pure. Hence, $R$ be a von-Neumann regular ring.
\end{proof}
\begin{proposition}
Let $M$ be a module which is fully invariant in its pure injective hull, then $M$ is pure extending module.
\end{proposition}
\begin{proof}  Proof follows from [ lemma 3.1,\cite{AH}].

\end{proof}
\begin{proposition}
    Let $R$ and $S$ be rings for which there is a Morita equivalence $F:$ Mod-$R\rightarrow$ Mod-$S$ and let $A\in$ Mod-$R$. Then $A$ is pure extending  iff $F(A)$ is pure extending.
\end{proposition}
\begin{proof}
    It follows from the morita invarient property of pure submodule.
\end{proof}
\begin{definition} 
     A submodule $K$ of a module $M$ is said to be pure essential in $M$ if $K$ is pure in $M$ and for any non zero submodule $N$ of $M$ either $K\cap N\neq 0$ or $(K\oplus N)/N$ is not pure in $M/N$ \cite{AH}.\\
\end{definition}
\begin{proposition}
     If every submodule of a module $M$ is pure essential in $M$, then $M$ is pure extending.
\end{proposition}
\begin{proof}
    Let $N$ be a pure submodule of $M$. Since every submodule is pure essential in $M$, $N$ is essential in $M$. Hence $M$ is pure extending module.
\end{proof}
    Converse of the above statement need not true in general. Now, we provide the sufficient condition when it holds true.
\begin{proposition}
    If $M$ is a pure simple pure extending module then every submodule of $M$ is pure essential submodule.
\end{proposition}
\begin{proof}
    If $M$ be a pure simple module then it has no non proper non trivial pure submodule. Hence every submodule of $M$ is pure essential.
\end{proof}
    
 \begin{corollary}
    If $M$ is an indecomposable pure extending module then every submodule of $M$ is pure essential submodule.
\end{corollary} 
\begin{proposition}
\label{prop1}
In any  pure quasi  injective module $M$, Every pure submodule of $M$ is essential in a direct  summand of $M$ .
\end{proposition}
\begin{proof}
Let $N$ be a pure submodule of  $M$ and write $PE(M)=PE(N)\oplus L$.
The pure quasi injectivity of $M$ implies $M\cap PE(M) = M \cap PE(N) \oplus M\cap L$ and $N{\le}^{e} M\cap PE(N)$ it implies pure extending.
\end{proof}

\begin{proposition}
Let $M$ be a $\mathbb{Z}$-module. If $M$ satisfies any one of the following conditions, then $M$ is pure extending module.
\begin{enumerate}
    \item[(i)] $M$ is finitely generated.
    \item[(ii)] $M$ is divisible.
\end{enumerate}
\end{proposition}
\begin{proof}

    $(i)$ Since every pure submodule of a finite generated module over a noetherian ring is a direct summand (by Corollary 4.91,\cite{TY}). Hence, $M$ is pure extending.\\
    $(ii)$ Let $M$ be a divisible module, this implies that it is an extending $\mathbb{Z}$ module. Hence $M$ is  pure extending module.

\end{proof}
   
    \begin{proposition}
    Every pure split module is pure extending.
    \end{proposition}
    \begin{proof}
    Let $M$ be a pure split $R$-module then every pure submodule of $M$ is direct summand. This implies  $M$ is a pure extending module.
    \end{proof}
    \begin{proposition} \label{thm3}
    Let $M$ be a module and $M=M_1\oplus M_2$ be direct sum decomposition . If $N$ be a pure  submodule of $M$ then $N=N_ 1 \oplus N_2$, where $N_i$ is pure submodule of $M_i$ for $i=1,2$.
    \end{proposition}
    \begin{proof} 
    Let $N$ be a pure submodule of $M$ such that $N=N_1+N_2$.  $N_1\cap N_2\leq N_1\leq M_1, N_1\cap N_2\leq N_2\leq M_2$ , which implies $N_1\cap N_2 \leq M_1 \cap M_2$, so $N_1\cap N_2 =0$. Hence, $N=N_1 \oplus N_2$. Since $N_i {\leq}^{\oplus} N$ then $N_i$ is pure submodule of $N$. $N_i$  is pure in $N$ and $N$ is pure in $M_i$ then $N_i$ is pure in $M_i$ for $i=1,2$.
    \end{proof}

In the next proposition, we show when direct sum of pure extending modules is pure extending module.

\begin{proposition} \label{dss}
Let $M=\bigoplus_{i\in \Lambda} M_i$, where $\Lambda$ be an arbitrary index set. $M$ is pure extending iff for each $i\in \Lambda$, $M_i$ is pure extending module. 
\end{proposition}
\begin{proof}
Let $M$ be a pure extending module. So, by proposition \ref{DS1.1} $M_i$ is pure extending module for each $i\in \Lambda$.\\
Conversely, Let $N$ be a pure submodule of $M$. So by proposition (\ref{thm3}), $N=\bigoplus _{i\in \Lambda} (N_i)$ such that every $N_i$ is a pure submodule of $M_i$ for each $i\in \Lambda$ . Since $M_i$ is pure extending,  there exists $X_i \leq ^{\oplus} M_i$ such that $N_i \leq ^{e} X_i$. Therefore $\bigoplus _{i\in \Lambda}(N_i)\leq ^{e} \bigoplus_{i\in \Lambda} X_i\leq ^{\oplus} M$. Hence $M$ is pure extending module.
\end{proof}

We say an $R$-Module $M$ is $RD$-pure extending if every $RD$-pure submodule is essential in a direct summand of $M$. 
\begin{lemma}\label{RDS1.1}
Every $RD$-pure extending module is pure extending.
\end{lemma}
\begin{proof}
Let $P$ be a pure submodule of $M$. Since every pure submodule is $RD$-pure and $M$ is $RD$-pure extending module. Therefore $P$ is essential in a direct summand of $M$. Hence the module $M$ is pure extending.
\end{proof}

\begin{example}
Let $M$ be an $R$ module such that every pure submodule is essential in a direct summand it implies $M$ is pure extending. In [page no.-158]\cite{TY}, there is an example which justify that every $RD$-pure submodule need not be pure submodule. So it may be possible that there exist $RD$-pure submodule of $M$ which are not essential in a direct summand of $M$. Hence $M$ need not be $RD$-pure extending.
\end{example}
\begin{proposition}
    \begin{enumerate}
        \item Direct summand of $RD$-pure extending module is $RD$-pure extending.
        \item  $RD$-pure submodule of $RD$-pure extending module is $RD$-pure extending.
    \end{enumerate}
\end{proposition}
\begin{proof}
The proof is similar to proposition \ref{psm} and \ref{DS1.1}. 
\end{proof}

\begin{proposition}
A flat $R$-module $M$ is pure extending iff $M$ is $RD$-pure extending.
\end{proposition}
\begin{proof}
The proof follows from lemma \ref{RDS1.1} and [Corollary 11.21,\cite{CF}].
\end{proof}
\begin{corollary}
A free (projective) $R$-module $M$ is pure extending iff $M$ is $RD$-pure extending.
\end{corollary}
\begin{corollary}
A faithful multiplicative $R$-module $M$ is pure extending iff $M$ is $RD$-pure extending.
\end{corollary}
\begin{proof}
The proof follows by the fact that every faithful multiplicative module is flat.
\end{proof}

\section{Characterization of rings using Pure extending modules}

In the next proposition we characterize the regular ring.
\begin{proposition}\label{fe1} For a ring $R$, the following conditions are equivalent:
\begin{enumerate}
        \item $R$ is a regular ring.
        \item Every pure extending $R$-module is flat.
    \end{enumerate}
    \end{proposition}
    \begin{proof}
    $(1)\Rightarrow (2)$ It is clear from [Theorem 4.21,\cite{TY}]\\
    $(2)\Rightarrow (1)$ Let $M$ be an right $R$-module and $PE(M)$ be a  pure injective hull of $M$. Then $0 \rightarrow M \rightarrow PE(M) \rightarrow PE(M)/M \rightarrow 0 $ is a pure exact sequence. From hypothesis, $PE(M)$ is a flat module so by [Corollary 4.86,\cite{TY}], $PE(M)/M$ is flat.  Therefore $M$ is a flat which implies $R$ is regular ring [Corollary 4.21,\cite{TY}].
    \end{proof}
    In the next proposition we characterize the semisimple ring.
     \begin{proposition}\label{fe2}
     For a ring $R$, the following conditions are equivalent:
     \begin{enumerate}
          \item $R$ is a right  semisimple ring.
          \item Every pure $C_3$ $R$-module is projective.
          \item Every pure $C_2$ $R$-module is projective
          \item Every quasi pure injective $R$-module is projective.
          \item Every pure injective $R$-module is projective.
          \item Every pure extending $R$-module is projective.
          \end{enumerate}
    \end{proposition}
    
     \begin{proof}
     $(1)\Rightarrow (2)$ Let $R$ be a right semisimple ring then every $R$-module $M$ is projective [Proposition 20.7,\cite{RW}]. So $(2)$ holds.\\
     $(2)\Rightarrow (3)$ Every pure $C_2$ module is pure $C_3$, so (3) holds.\\
     $(3)\Rightarrow (4)$ Every quasi pure Injective right $R$- module is pure $C_2$, so (4) holds.\\
    $(4)\Rightarrow (5)$ Every pure injective right $R$-module  is quasi pure injective. so (5) holds.\\
    $(5)\Rightarrow (1)$ Every pure injective right $R$-module is projective it implies every injective is projective. Therefore $(1)$ holds by [Proposition 20.7,\cite{RW}] \\
     $(1) \Leftrightarrow (4)$ It follows from [Proposition 6,\cite{SK}].\\
     $(1)\Rightarrow (6)$ Let $R$ be a semisimple ring it implies every $R$-module is injective (in particular pure extending) and projective [Proposition 20.7,\cite{RW}], so (6) holds.\\
     $(6)\Rightarrow (1)$ By hypothesis, every pure extending  module is projective it implies injective is projective. Hence, $R$ is a right semisimple ring \cite[Proposition 20.7]{RW}.
     \end{proof}
     \begin{proposition}
Let $R$ be a local ring then the module $R_R$ is pure extending.
\end{proposition}
\begin{proof}
By [Theorem 3,\cite{FDJ}], every local ring is pure simple it implies  $R_R$ and 0  are its only pure submodules. Hence, $R_R$ is pure extending module.
\end{proof}
      A ring $R$ is called left PDS if pure submodule of left $R$-module are direct summand and ring $R$ is said to be a PDS if it is both left and right PDS \cite{FH}.
\begin{proposition}
    For a PDS ring $R$, every $R$-module is pure extending.
\end{proposition}
\begin{proof}
     Let $M$ be an $R$-module. By hypothesis, $R$ is PDS ring  which implies every pure submodule is essential in a direct summand. Hence $M$ be pure extending module.
\end{proof}
\begin{proposition}\label{cm}
  Every flat cotorsion module is pure extending.
\end{proposition}
    \begin{proof}
    Let $M$ be a flat and cotorsion module then by [Proposition 3.2,\cite{LM}], $M$ is quasi pure injective. Therefore $M$ is a pure extending module.
\end{proof}
\begin{corollary}
     If $R$ is a right cotorsion ring then $R_R$ is pure extending $R$-module.
\end{corollary}


\begin{thebibliography}{99}
\bibitem{MMA}
M.M. Ali, and D.J. Smith, \textit{Pure submodules of multiplication modules}, Beiträge zur Algebra und Geometrie \textbf{45(1)}(2004), 61-74.
\bibitem{AF}
F.W. Anderson, and K.R. Fuller, \textit{Rings and categories of modules}, Springer Science and Business Media \textbf{Vol. 13}(2012).
\bibitem{CK}
A.W. Chatters, and S.M. Khuri, \textit{Endomorphism rings of modules over non-singular CS rings}, Journal of the London Mathematical Society \textbf{2(3)}(1980), 434-444.
\bibitem{CAW}
A.W. Chatters, and C.R. Hajarnavis, \textit{Rings in which every complement right ideal is a direct summand}, The Quarterly Journal of Mathematics \textbf{28(1)}(1977), 61-80.
\bibitem{PMC}
P.M. Cohn, \textit{On the free product of associative rings}, Mathematische Zeitschrift \textbf{71(1)}(1959), 380-398. 
\bibitem{MD}
A.M. Dehkordi, \textit{On the structure of pure-projective modules and some applications}, Journal of pure and Applied Algebra \textbf{221(4)}(2017), 935-947.
\bibitem{EE}
 E.E. Enochs, \textit{Flat covers and flat cotorsion modules}, Proc. Amer. Math Soc. \textbf{92}(1984), 179-184.
 \bibitem{CF}
C. Faith, \textit{Algebra: Rings, Modules and Categories}, Springer-Verlag (1981).
\bibitem {FH}
D.J. Fieldhouse, \textit{Pure theories. Math. Ann.} \textbf{184}(1969), 71-78.
\bibitem{FDJ}
D.J. Fieldhouse, \textit{Pure simple and indecomposable rings}, Canadian Mathematical Bulletin \textbf{13(1)}(1970), 71-78.
\bibitem{AH}
A. Harmanci, S.R. López-Permouth, And B. Ungor, \textit{On the pure-injectivity profile of a ring}, Communications in Algebra \textbf{43(11)}(2015), 4984-5002.  
\bibitem{TY}
T.Y. Lam, \textit{Lectures on modules and rings}, Springer Science and Business Media \textbf{Vol.189}(2012).
\bibitem{LM} 
L. Mao, and N. Ding, \textit{Cotorsion modules and relative pure-injectivity}, Journal of the Australian Mathematical Society \textbf{ 81(2)}(2006), 225-244.
\bibitem{SK}
 S.K. Maurya, S. Das, and Y. Alagoz, \textit{Pure-Direct-Injective Modules}, Lobachevskii Journal of Mathematics \textbf{43(2)}(2022), 416-428.
\bibitem{MM}
 S.H. Mohamed, and B.J. Müller, \textit{Continuous and discrete modules}, Cambridge University Press \textbf{Vol. 147}(1990).
\bibitem{PR}
P. Ribenboim, \textit{Algebraic numbers. John Wiley and Sons} \textbf{Vol.27}(1972).
\bibitem{PV}
 T. Sharpe, D.W. Sharpe, and P. Vámos, \textit{Injective modules}, Cambridge University Press \textbf{Vol.62}(1972).
  \bibitem{YU}
 Y.Utumi, \textit{On Continuous rings and selfinjective rings}, Trans. Amer. Math Soc.\textbf{70}(1965), 158-173.
\bibitem{WF}
 R. Warfield, \textit{Purity and algebraic compactness for modules}, Pacific Journal of Mathematics \textbf{28(3)}(1969), 699-719.

\bibitem{RW}
R. Wisbauer, \textit{Foundations of Module and Ring Theory: A handbook for study and research Routledge},(2018).
\end{thebibliography}
\end{document}